\documentclass{article}

\usepackage{amsmath,amsthm,amscd,amsfonts,amssymb}

\usepackage{a4wide}

\newtheorem{theorem}{Theorem}[section]
\newtheorem{lemma}[theorem]{Lemma}
\newtheorem{proposition}[theorem]{Proposition}
\theoremstyle{definition}
\newtheorem{definition}[theorem]{Definition}

\begin{document}

\title{Optimal control with time-delays via the penalty method\thanks{This is a preprint of a paper 
whose final and definite form is: Mathematical Problems in Engineering (ISSN 1024-123X) 2014, 
Article ID 250419, http://dx.doi.org/10.1155/2014/250419}}

\author{Mohammed Benharrat$^1$\\
\texttt{mohammed.benharrat@gmail.com}
\and Delfim F. M. Torres$^2$\\
\texttt{delfim@ua.pt}}

\date{$^1$D\'epartement de Math\'ematiques et Informatique,\\
Ecole National Polytechnique d'Oran (Ex. ENSET d'Oran),\\
B.P. 1523 El M'Naouar, Oran, Alg\'{e}rie\\[0.3cm]
$^2$Center for Research and Development in Mathematics and Applications (CIDMA),
Department of Mathematics, University of Aveiro,\\ 3810--193 Aveiro, Portugal}


\maketitle


\begin{abstract}
We prove necessary optimality conditions of Euler--Lagrange type for
a problem of the calculus of variations with time delays, where the
delay in the unknown function is different from the delay in its derivative.
Then, a more general optimal control problem with time delays
is considered. Main result gives a convergence theorem, allowing to obtain
a solution to the delayed optimal control problem by considering
a sequence of delayed problems of the calculus of variations.

\bigskip


\end{abstract}


\section{Introduction}

Over the past years, there has been an increasing interest in time-delay
problems of the calculus of variations and control \cite{39a,MyID:256,MyID:253,book}.
Such interest is explained for their importance in control and engineering
\cite{MyID:264,MyID:298,MR2141765,29:Salamon}. Indeed, time delays are inherent
in various real systems, such as control systems and optimal control
problems in engineering \cite{MR2553382,MR3124697}.

In this paper we improve recent optimality conditions
for time-delay variational problems. In \cite{MyID:231}
necessary optimality conditions of Euler--Lagrange,
DuBois--Reymond and Noether type were obtained for problems
of the calculus of variations with a time delay.
The results of \cite{MyID:231} were then extended to
delayed variational problems with higher order
derivatives in \cite{MyID:256}. Here we model
time-delay variational problems in a more realistic way:
while in \cite{MyID:256,MyID:231} the delay on
functions and their derivatives (and control variables)
is always the same, here we consider different delays for the functions
and derivatives/controls.

The text is organized as follows. In Section~\ref{sub1}
we formulate the delayed problem of the calculus of variations,
where the delay in the unknown functions
is different from the delay in their derivatives.
The main result in this section is Theorem~\ref{th:EL1},
which provides necessary optimality conditions of Euler--Lagrange type.
Control strategies via an exterior penalty method are then investigated
in Section~\ref{sec:oc:td}.
The idea is to replace the optimal control problem with time-delays
by a series of delayed problems of the calculus of variations.
The main result gives a convergence theorem that allows
to obtain a solution to delayed optimal control problems
with linear delayed control systems, by considering a sequence
of variational problems with time-delays of the type considered before
in Section~\ref{sub1} (see Theorem~\ref{thm4.1}).
We end with Section~\ref{sec:conc} of conclusions.


\section{Calculus of variations with time delays}
\label{sub1}

We consider the following fundamental problem of the calculus of variations
with time delays, where the delay in the function we are looking for
is different from the delay in its derivative:
\begin{equation}
\label{Pe}
\min \int_{0}^{\top}
L\left(t,x(t),x(t-\tau_1),\dot{x}(t),\dot{x}(t-\tau_2)\right) dt
\end{equation}
subject to
\begin{equation}
\label{Pe2}
\begin{split}
x(t)&=\theta_{1}(t), \quad t\in[-\tau_{1},-\tau_{2}] =: I_1,\\
x(t)&=\theta_{2}(t), \quad t\in[-\tau_{2},0] =: I_2,
\end{split}
\end{equation}
and
\begin{equation}
\label{Pe4}
x(\top)=\alpha,
\end{equation}
where $L :[0, \top] \times \mathbb{R}^{4 N} \rightarrow \mathbb{R}$,
$\left( t,a,\bar{a},b,\bar{b}\right) \rightarrow L\left( t,a,\bar{a},
b,\bar{b}\right)$, is the Lagrangian, $\top > 0$ is fixed in $\mathbb{R}$,
$\tau_{1}$ and  $\tau_{2}$ are two given positives real numbers such that $\tau_{2}< \tau_{1}<\top$,
and $\theta_{1}(\cdot)$ and $\theta_{2}(\cdot)$ are given piecewise smooth functions.
Let $I:=\left[ 0,\top \right]$, $L^{2}\left( I ,\mathbb{R}^{N}\right)$ be the Lebesgue space of measurable functions
such that
$$
\left\|x\right\|_{L^{2}}= \left(
\int^{\top}_{0}{\left\| x\left( t \right) \right\|_{\mathbb{R}^{N}}^2 dt}\right)^{\frac{1}{2}} <\infty
$$
and $H^{1}\left( I,\mathbb{R}^{N}\right)$ be the Sobolev space
of functions having their weak first derivative lying in $L^{2}\left( I,\mathbb{R}^{N}\right)$
and represented by
$$
x\left(t\right) =x(\tau)+\int\limits_{{\tau}}^{t}\dot{x}\left(s \right) ds
$$
for all $\tau$ and $t$ in $ I$. We denote
\begin{itemize}
\item $\mathcal{H}$ the space of all functions
$x : \left[ -\tau_{1},\top\right] \rightarrow {\mathbb{R}^{N}}$ such that
$x_{/ I_{1}}\in {L}^{2}\left( I_{1}, \mathbb{R}^{N}\right)$, $x _{/I_{2}}\in
H^{1}\left(  I_{2}, \mathbb{R}^{N}\right)$  and $x_{/ I}\in  H^{1}(  I ,\mathbb{R}^{N})$,
which is a Hilbert space with the norm
$$
\left\Vert x \right\Vert_{\mathcal{H}}=\left( \left\Vert x _{/ I_{1}}\right\Vert_{ L^{2}(
I_{1},\mathbb{R}^{N}) }^{2} +\left\Vert x_{/ I_{2}}\right\Vert _{ H^{1}(
I_{2}, \mathbb{R}^{N})}^{2}+\left\Vert x_{/ I}\right\Vert _{H^{1}\left(
I,\mathbb{R}^{N}\right)}^{2}\right) ^{\frac{1}{2}};
$$
\item $D:=\left\{ x\left(\cdot\right) \in \mathcal{H} : x_{/ I_{1}}
=\theta_{1}, x_{/I_{2}}=\theta_{2} ,\text{ and   }x\left( \top\right)
=\alpha\right\}$;
\item $J: \mathcal{H} \longrightarrow \mathbb{R}$ the functional
$$
J\left( x\left(\cdot\right) \right)
=\int_{0}^{\top}L\left( t,x\left( t\right), x\left( t-\tau_{1}\right), \dot{x}\left( t\right),
\dot{x}\left( t-\tau_{2}\right) \right) dt.
$$
\end{itemize}
Our problem \eqref{Pe}--\eqref{Pe4} takes then the following form:
\begin{equation}
\label{eq:prb}
\min_{x\left( \cdot\right) \in D} J\left( x\left( \cdot \right) \right).
\end{equation}
We make the following assumptions on the data of problem \eqref{eq:prb}:
\begin{itemize}
\item[$(A_1)$] Lagrangian $L$ is a $C^{1}$ Carath\'{e}odory mapping, i.e.,
it is of class $C^{1}$ in $\left( a,\bar{a},b,\bar{b}\right)$ for almost all
$t \in \left[ 0,\top\right]$ and is measurable in $t$ for every
$\left( a,\bar{a},b,\bar{b}\right)$;
\item[$(A_2)$] there exist $\gamma_i \left(\cdot\right)
\in  L^{2}\left( I,\mathbb{R}^{+}\right)$, $i=1, \ldots, 5$,
such that a.e. in $t\in I$
\begin{equation*}
\begin{split}
\left\vert L\left( t,a,\bar{a},b,\bar{b}\right) \right\vert
&\leq \gamma_1 \left( t\right),\\
\left\Vert \partial_{i}L\left( t,a,\bar{a},b,\bar{b}\right) \right\Vert
&\leq \gamma_{i} \left( t\right), \quad i=2, \ldots, 5,
\end{split}
\end{equation*}
where $\partial_{i}L$ is the partial derivative of $L$
with respect to its $i$th argument.
\end{itemize}

\begin{definition}[Cone of tangents]
\label{def:cone:tang}
Let $Z$ be a normed space, $A\subset Z$, and $a\in \overline{A}$. The cone of tangents
$T(A,a)$ is the set of all $z\in Z$ with the property that there is a sequence $(a_n)$
in $A$ converging strongly to $a$ and a sequence of non-negative numbers $(\alpha_n )$
such that $\alpha_n (a_n -a)\rightarrow z$.
\end{definition}

\begin{lemma}
\label{l41}
The  set $D$ is an affine linear subspace of  $\mathcal{H}$ and the cone of tangents $T(D,x(\cdot))$ is given by
$$
T\left( D,x\left(\cdot\right) \right) =\left\{ v\left(\cdot\right)
\in \mathcal{H}: v\left(\cdot\right)_{/ I_{1}}=0,\
v\left(\cdot\right)_{/I_{2}}=0,\text{ and }v\left( \top\right) =0\right\}.
$$
\end{lemma}

\begin{proof}
Let $v\left(\cdot\right) \in T\left( D,x\left(\cdot\right)\right)$.
Then there exist $\left( x_{n}\left(\cdot\right) \right)_{n}\subset D$ and  $\lambda _{n}\geq 0$
such that $x_{n}\left(\cdot\right) \rightarrow x\left(\cdot\right)$
in  $D$ implies that $ \lambda _{n}\left( x_{n}\left(\cdot\right)-x\left(\cdot\right)\right)
\rightarrow v\left(\cdot\right)$ in $\mathcal{H}$. Since
$x_{n}\left(\cdot\right),x\left(\cdot\right) \in  D$ for all $n$, we have
$$
\begin{array}{l}
x_{n}\left( \tau \right) =x\left( \tau \right)
=\theta _{1}\left( \tau\right),
\quad \tau \in \left[ -\tau_{1},-\tau_{2}\right],\\
x_{n}\left( \tau \right) =x\left( \tau \right)
=\theta _{2}\left( \tau\right),
\quad \tau \in \left[ -\tau_{2},0\right],\\
x_{n}\left(\top\right) =x\left( \top\right) =\alpha.
\end{array}
$$
Hence,
$$
\begin{array}{l}
\lambda _{n}\left( x_{n}\left( \tau \right) -x\left( \tau \right) \right) =0,
\quad \tau \in \left[ -\tau_{1},-\tau_{2}\right],\\
\lambda _{n}\left( x_{n}\left( \tau \right) -x\left( \tau \right) \right) =0,
\quad \tau \in \left[ -\tau_{2},0\right],\\
\lambda _{n}\left( x_{n}\left( \top\right) -x\left( \top\right) \right) =0.
\end{array}
$$
Therefore, $v(\cdot)\in \mathcal{H}$ with  $v(\tau )=0$
for all $\tau \in \left[ -\tau_{1},-\tau_{2}\right]$,
$v(\tau )=0$  for almost all $\tau \in \left[ -\tau_{2},0\right]$,
and $v(\top)=0$. Thus,
$$
T\left(D,x(\cdot)\right) \subset \left\{ v \left(\cdot\right) \in \mathcal{H} :
v(\cdot)_{/I_{1}} =0, v\left(\cdot\right)_{/I_{2}}=0,
\text{ and } v\left( \top \right) =\alpha \right\}=K.
$$
Conversely, let $v(\cdot)\in K$ for $x\left(\cdot\right) \in D$.
Define $x_{n}\left(\cdot\right)
=x\left(\cdot\right) +\frac{1}{n}v\left(\cdot\right)$.
Then $n\left( x_{n}\left(\cdot\right) -x\left(\cdot\right) \right)
=v\left(\cdot\right)$ with $v(\cdot) \in \mathcal{H}$. Hence,
$v(\cdot)\in T\left( D,x\left(\cdot\right) \right)$.
\end{proof}

For convenience, we introduce the operator $[\cdot]_{\tau_{1}}^{\tau_{2}}$ defined by
$$
[x]_{\tau_{1}}^{\tau_{2}}(t)
=\left(t,x(t),\dot{x}(t),x(t-\tau_1 ),\dot{x}(t-\tau_2)\right).
$$

\begin{proposition}
Under conditions $(A_1)$ and $(A_2)$, the mapping
$J\left(\cdot\right)$ is Fr\'{e}chet differentiable and
\begin{multline*}
J^{\prime }\left( x\left(\cdot\right) \right) \left( v\left(\cdot\right)\right)
=\int_{0}^{\top}\biggl(
\left\langle \partial_{2}L[x]_{\tau_{1}}^{\tau_{2}}(t),v\left( t\right)
\right\rangle
+\left\langle \partial_{3}L[x]_{\tau_{1}}^{\tau_{2}}(t),v\left( t-\tau_{1}\right)
\right\rangle \\
+\left\langle \partial_{4}L[x]_{\tau_{1}}^{\tau_{2}}(t),\dot{v}\left( t\right) \right\rangle
+\left\langle \partial_{5}L[x]_{\tau_{1}}^{\tau_{2}}(t), \dot{v}\left( t-\tau_{2}\right)
\right\rangle \biggr) dt.
\end{multline*}
\end{proposition}

\begin{proof}
Let $v\left(\cdot\right) \in \mathcal{H}$. We have
\begin{align*}
J^{\prime } \left( x\left(\cdot\right) ;v\left(\cdot\right) \right)
&= \underset{\lambda \rightarrow 0^{+}}{\lim }\frac{1}{\lambda }\left( J\left( x\left(\cdot\right)
+\lambda v\left(\cdot\right) \right) -J\left(
x\left(\cdot\right) \right) \right) \\
 &= \underset{\lambda \rightarrow 0^{+}}{\lim }
\int_{0}^{\top}\frac{1}{\lambda }\left[ L[x+\lambda v]_{\tau_{1}}^{\tau_{2}}(t) -L[x]_{\tau_{1}}^{\tau_{2}}(t)\right] dt.
\end{align*}
Define
$$
\Psi _{\lambda }\left( t\right) =\underset{\lambda \rightarrow 0^{+}}{\lim }
\int_{0}^{\top}\frac{1}{\lambda }\left[ L[x+\lambda v]_{\tau_{1}}^{\tau_{2}}(t) -L[x]_{\tau_{1}}^{\tau_{2}}(t)\right] dt
$$
and
\begin{multline*}
\Psi \left( t\right) =\left\langle \partial_{2}L[x]_{\tau_{1}}^{\tau_{2}}(t),v\left( t\right)
\right\rangle
+\left\langle \partial_{3}L[x]_{\tau_{1}}^{\tau_{2}}(t),v\left( t-\tau_{1}\right)
\right\rangle\\
+\left\langle \partial_{4}L[x]_{\tau_{1}}^{\tau_{2}}(t),\dot{v}\left( t\right) \right\rangle
+\left\langle \partial_{5}L[x]_{\tau_{1}}^{\tau_{2}}(t), \dot{v}\left( t-\tau_{2}\right)
\right\rangle.
\end{multline*}
Then, $\Psi _{\lambda }\left( t\right) \rightarrow \Psi\left( t\right)$ as $\lambda \rightarrow 0^{+}$
for almost all $t\in \left[ 0,\top\right]$. On the other hand,
$\left\vert \Psi _{\lambda }\left( t\right) \right\vert
\leq g\left( t\right)$  a.e. in $t\in \left[ 0,\top\right]$ with
$$
g\left( t\right) =\gamma \left( t\right) \left[ \left\Vert v\left(
t\right) \right\Vert _{{\mathbb{R}^{N}}}+\left\Vert v\left( t-\tau_{1}\right)
\right\Vert _{{\mathbb{R}^{N}}}+\left\Vert
\dot{v}\left( t\right) \right\Vert_{{\mathbb{R}^{N}}}
+\left\Vert v\left( t-\tau_{2}\right)\right\Vert_{{\mathbb{R}^{N}}}\right]
$$
a function not depending on $\lambda $, and
$\left\vert \Psi _{\lambda }\left( t\right) \right\vert \leq g\left( t\right)+1$
for almost all $t\in \left[ 0,\top\right]$ and $\lambda$ sufficiently small.
Since $\left[ 0,\top\right]$ has finite measure,
Lebesgue's theorem yields that $\int_{0}^{\top}\Psi _{\lambda}\left( t\right) dt
\rightarrow \int_{0}^{\top}\Psi \left(t\right) dt$ as
$\lambda \rightarrow 0^{+}$. Hence,
\begin{multline*}
J^{\prime }\left( x\left(\cdot\right) \right) \left( v\left(\cdot\right)\right)
=\int_{0}^{\top} \Bigl(\left\langle \partial_{2}L[x]_{\tau_{1}}^{\tau_{2}}(t),v\left( t\right)
\right\rangle+\left\langle \partial_{3}
L[x]_{\tau_{1}}^{\tau_{2}}(t),v\left( t-\tau_{1}\right)\right\rangle \\
+\left\langle \partial_{4}L[x]_{\tau_{1}}^{\tau_{2}}(t),\dot{v}\left( t\right) \right\rangle
+\left\langle \partial_{5}L[x]_{\tau_{1}}^{\tau_{2}}(t), \dot{v}\left( t-\tau_{2}\right)
\right\rangle \Bigr) dt.
\end{multline*}
This is the directional derivative of $J$ in the direction $v$.
To finish the proof, we need to show that
$J^{\prime}\left( x\left(\cdot\right) ;v\left(\cdot\right) \right)$
is linear and bounded in $v$ and continuous in $x$. The linearity is obvious.
We begin by proving that $J^{\prime }\left(x(\cdot);\cdot\right)$
is bounded from $\mathcal{H}$ to $\mathbb{R}$:
\begin{align*}
\left\vert J^{\prime }\left( x\left(\cdot\right);
v\left(\cdot\right)\right) \right\vert
&\leq \int_{0}^{\top}\left\vert  \left\langle \partial_{2}
L[x]_{\tau_{1}}^{\tau_{2}}(t),v\left( t\right)\right\rangle  \right\vert dt
+\int_{0}^{\top}\left\vert \left\langle \partial_{3}
L[x]_{\tau_{1}}^{\tau_{2}}(t),v\left( t-\tau_{1}\right)
\right\rangle \right\vert dt\\
&\quad + \int_{0}^{\top}\left\vert \left\langle \partial_{4}
L[x]_{\tau_{1}}^{\tau_{2}}(t),\dot{v}\left( t\right) \right\rangle \right\vert dt
+\int_{0}^{\top}\left\vert \left\langle \partial_{5}
L[x]_{\tau_{1}}^{\tau_{2}}(t), \dot{v}\left( t-\tau_{2}\right)
\right\rangle\right\vert dt\\
& \leq \int_{0}^{\top}\gamma_2 \left( t\right) \left\Vert
v\left( t\right) \right\Vert _{{\mathbb{R}^{N}}}dt+\int_{0}^{\top
}\gamma_3 \left( t\right) \left\Vert v\left(t
-\tau_{1}\right) \right\Vert_{{\mathbb{R}^{N}}}dt\\
& \quad +\int_{0}^{\top}\gamma_4 \left( t\right)
\left\Vert \dot{v}\left( t\right)\right\Vert_{{\mathbb{R}^{N}}}dt
+\int_{0}^{\top}\gamma_5 \left( t\right) \left\Vert \dot{v}\left(
t-\tau_{2}\right) \right\Vert _{{\mathbb{R}^{N}}}dt\\
& \leq \int_{0}^{\top}\gamma_2 \left( t\right) \left\Vert
v\left( t\right) \right\Vert _{{\mathbb{R}^{N}}}dt
+\int_{-\tau_{1}}^{\top -\tau_{1}}\gamma_3 \left( t+\tau_{1}\right) \left\Vert
v\left(t\right) \right\Vert _{{\mathbb{R}^{N}}}dt\\
&\quad +\int_{0}^{\top}\gamma_4
\left( t\right) \left\Vert \dot{v}\left( t\right) \right\Vert_{{\mathbb{R}^{N}}}dt
+\int_{-\tau_{2}}^{\top-\tau_{2}}\gamma_5 \left( t+\tau_{2}\right)
\left\Vert \dot{v}\left( t\right) \right\Vert _{{\mathbb{R}^{N}}}dt \\
& \leq \int_{0}^{\top}\gamma_2 \left( t\right) \left\Vert
v\left( t\right) \right\Vert _{{\mathbb{R}^{N}}}dt+\int_{-\tau_{1}}^{
-\tau_{2}}\gamma_3 \left( t+\tau_{1}\right) \left\Vert v\left( t\right)
\right\Vert _{{\mathbb{R}^{N}}}dt\\
&\quad +\int_{-\tau_{2}}^{0}\gamma_3 \left( t+\tau_{1}\right)
\left\Vert v\left( t\right) \right\Vert _{{\mathbb{R}^{N}}}dt
+ \int_{0}^{ \top -\tau_{1}} \gamma_3 \left(t+\tau_{1}\right)
\left\Vert v\left( t\right) \right\Vert_{\mathbb{R}^{N}}dt\\
& \quad +\int_{0}^{\top}\gamma_4 \left( t\right)
\left\Vert \dot{v}\left( t\right) \right\Vert _{{\mathbb{R}^{N}}}dt
+\int_{-\tau_{2}}^{0}\gamma_5 \left( t+\tau_{2}\right)
\left\Vert \dot{v}\left( t\right) \right\Vert _{{\mathbb{R}^{N}}}dt\\
& \leq M\left\Vert v\left(\cdot\right) \right\Vert_{\mathcal{H}}.
\end{align*}
We still need to prove the continuity of $J^{\prime }\left(\cdot\right)$.
Let $x_{n}\left(\cdot\right) \rightarrow x\left(\cdot\right)$
in $\mathcal{H}$. Then,
\begin{equation*}
\begin{split}
\biggl\vert \bigl[ J^{\prime }\left( x_{n}\left(\cdot\right)\right)
&-J^{\prime }\left( x\left(\cdot\right) \right) \bigr] \left(
v\left(\cdot\right) \right) \biggr\vert\\
& \leq \int_{0}^{\top}\left\vert
\left\langle \partial_{2}L[x_n -x]_{\tau_{1}}^{\tau_{2}}(t),v\left( t\right)
\right\rangle  \right\vert dt
+\int_{0}^{\top}\left\vert \left\langle \partial_{3}
L[x_n -x]_{\tau_{1}}^{\tau_{2}}(t),v\left( t-\tau_{1}\right)
\right\rangle \right\vert dt\\
& \quad + \int_{0}^{\top}\left\vert \left\langle \partial_{4}
L[x_n -x]_{\tau_{1}}^{\tau_{2}}(t),\dot{v}\left( t\right) \right\rangle \right\vert dt
+\int_{0}^{\top}\left\vert \left\langle \partial_{5}
L[x_n -x]_{\tau_{1}}^{\tau_{2}}(t), \dot{v}\left( t-\tau_{2}\right)
\right\rangle\right\vert dt\\
& \leq  \int_{0}^{\top} \left\| \partial_{2}L[x_n -x]_{\tau_{1}}^{\tau_{2}}(t)\right\| \left\Vert
v\left( t\right) \right\Vert _{{\mathbb{R}^{N}}}dt
+\int_{0}^{\top} \left\| \partial_{3}
L[x_n -x]_{\tau_{1}}^{\tau_{2}}(t)\right\|\left\Vert v\left( t-\tau_{1}\right) \right\Vert _{
{\mathbb{R}^{N}}}dt \\
& \quad +\int_{0}^{\top} \left\| \partial_{4}L[x_n -x]_{\tau_{1}}^{\tau_{2}}(t)\right\|
\left\Vert \dot{v}\left( t\right) \right\Vert _{{\mathbb{R}^{N}}}dt
+\int_{0}^{\top} \left\| \partial_{5}
L[x_n -x]_{\tau_{1}}^{\tau_{2}}(t)\right\| \left\Vert \dot{v}\left(
t-\tau_{2}\right) \right\Vert_{{\mathbb{R}^{N}}}dt \\
&= I_{2}+I_{3}+I_{4}+I_{5},
\end{split}
\end{equation*}
where
\begin{align*}
I_{2}& \leq \left\| \partial_{2}L[x_n -x]_{\tau_{1}}^{\tau_{2}}(\cdot)\right\| \left\Vert
v\left( t\right) \right\Vert _{L^{2}},\\
I_{3} & \leq \left\| \partial_{3}L[x_n -x]_{\tau_{1}}^{\tau_{2}}(\cdot)\right\| \left\Vert
v\left( t\right) \right\Vert _{L^{2}([-\tau_1 , \top], \mathbb{R}^{N})},\\
	I_{4} & \leq\left\| \partial_{4}L[x_n -x]_{\tau_{1}}^{\tau_{2}}(\cdot)\right\| \left\Vert
\dot{v}\left( t\right) \right\Vert _{L^{2}},\\
I_{5} & \leq \left\| \partial_{3}L[x_n -x]_{\tau_{1}}^{\tau_{2}}(\cdot)\right\| \left\Vert
\dot{v}\left( t\right) \right\Vert _{L^{2}([-\tau_1 , \top], \mathbb{R}^{N})}.
\end{align*}
On the other hand,
$x_{n}\left(\cdot\right) \rightarrow x\left(\cdot\right) $ in $\mathcal{H}$.
From Lebesgue's theorem, there exists $\Bbbk _{1},\Bbbk _{2},\Bbbk _{3}
\subset \mathbb{N}$ such that  $\Bbbk _{1}\subset \Bbbk _{2}\subset \Bbbk _{3}$ and
\begin{gather*}
x_{k}\left( t\right) \rightarrow x\left( t\right),
\quad \text{ a.e. } t\in \left[ 0, \top \right], \text{ for all } k\in \Bbbk_{1},\\
\dot{x}_{k}\left( t\right) \rightarrow \dot{x}\left( t\right),
\quad \text{ a.e. } t\in \left[ 0, \top \right], \text{ for all } k\in \Bbbk _{1},\\
x_{k}\left( t\right) \rightarrow x\left( t\right),
\quad \text{ a.e. } t \in [ -\tau_{2}, 0], \text{ for all } k\in \Bbbk _{2},\\
\dot{x}_{k}\left( t\right) \rightarrow \dot{x}\left( t\right),
\quad \text{ a.e. } t\in [ -\tau_{2}, 0], \text{ for all } k \in \Bbbk _{2},\\
x_{k}\left( t\right) \rightarrow x\left( t\right)
\quad \text{ a.e. } t\in [-\tau_{1}, -\tau_{2}], \text{ for all } k\in \Bbbk _{3}.
\end{gather*}
Hence,
\begin{gather*}
\dot{x}_{k}\left( t-\tau_{2}\right) \rightarrow \dot{x}\left( t-\tau_{2}\right),
\quad \text{ a.e. } t\in \left[ 0,\tau_{2}\right], \text{ for all } k \in \Bbbk _{2},\\
x_{k}\left( t-\tau_{1}\right) \rightarrow x\left( t-\tau_{1}\right),
\quad  \text{ a.e. } t \in \left[0,\tau_{1}\right], \text{ for all } k \in \Bbbk _{3}.
\end{gather*}
Since $L\left(t,\cdot,\cdot,\cdot\right)$ is $C^{1}$-Carath\'{e}odory,
assumption $(A_2)$ assures from Lebesgue's theorem that
$$
\left\| \partial_{i}L[x_n -x]_{\tau_{1}}^{\tau_{2}}(\cdot)\right\|
\longrightarrow 0, \quad i=2,3,4,5.
$$
This implies that $I_{1}+I_{2}+I_{3}+I_{4}\rightarrow 0$. Then,
$J^{\prime }\left(x_{k}\left(\cdot\right) \right) \rightarrow  J^{\prime }(x(\cdot))$.
The proof is complete.
\end{proof}

\begin{theorem}[Necessary optimality conditions of Euler--Lagrange type for problem \eqref{Pe}--\eqref{Pe4}]
\label{th:EL1}
Under conditions $(A_1)$ and $(A_2)$,
if $\bar{x}(\cdot)$ is a minimizer to problem \eqref{Pe}--\eqref{Pe4},
then $\bar{x}(\cdot)$ satisfies the following
\emph{Euler--Lagrange equations with time delay}:
\begin{equation*}
\begin{cases}
\frac{d}{dt}\left\{\partial_{4}L[\bar{x}]_{\tau_{1}}^{\tau_{2}}(t)+
\partial_{5}L[\bar{x}]_{\tau_{1}}^{\tau_{2}}(t+\tau_{2})\right\}
=\partial_{2}L[\bar{x}]_{\tau_{1}}^{\tau_{2}}(t)+\partial_{3}L[\bar{x}]_{\tau_{1}}^{\tau_{2}}(t+\tau_{1}),
\quad \text{ a.e. } t\in \left[ 0, \top -\tau_{1}\right],\\
\frac{d}{dt}\left\{\partial_{4}L[\bar{x}]_{\tau_{1}}^{\tau_{2}}(t)
+\partial_{5}L[\bar{x}]_{\tau_{1}}^{\tau_{2}}(t+\tau_{2})\right\}
=\partial_{2}L[\bar{x}]_{\tau_{1}}^{\tau_{2}}(t),
\quad \text{ a.e. } t \in \, ] \top-\tau_{1}, \top -\tau_{2}],\\
\frac{d}{dt}\partial_{4}L[\bar{x}]_{\tau_{1}}^{\tau_{2}}(t) =\partial_{2}L[\bar{x}]_{\tau_{1}}^{\tau_{2}}(t),
\quad  \text{ a.e. } t\in \, ]\top-\tau_{2}, \top].
\end{cases}
\end{equation*}
\end{theorem}

\begin{proof}
If $\bar{x}\left(\cdot\right)$ is a minimizer to problem \eqref{Pe}--\eqref{Pe4}, then
\begin{equation*}
J^{\prime }\left( \bar{x}\left(\cdot\right) \right) \left( v\left(\cdot\right) \right) =0
\end{equation*}
for all $v\left(\cdot\right) \in T\left(D,\bar{x}\left(\cdot\right) \right)$, that is,
\begin{equation}
\label{equ10}
\int_{0}^{ \top }\left( \left\langle p_2 (t), v\left( t\right) \right\rangle
+\left\langle  p_3 (t), v\left( t-\tau_{1}\right) \right\rangle
+\left\langle  p_4 (t),\dot{v}\left( t\right)
\right\rangle +\left\langle  p_5 (t), \dot{v}\left( t-\tau_{2}\right) \right\rangle \right) dt=0
\end{equation}
for all $v\left(\cdot\right) \in T\left( D,\bar{x}\left(\cdot\right) \right)$
with $p_{i}(t)=\partial_{i}L[\bar{x}]_{\tau_{1}}^{\tau_{2}}(t)$, $i=2,3,4,5$.
Integration by parts yields
\begin{equation}
\label{equ20}
\int_{0}^{ \top }\left\langle p_4 \left( t\right), \dot{v}\left( t\right) \right\rangle dt=-\int_{0}^{
\top }\left\langle \dot{p}_4 \left( t\right), v\left(
t\right) \right\rangle dt
\end{equation}
and
\begin{equation}
\label{equ30}
\int_{0}^{ \top }\left\langle p_5 (t), v\left( t-\tau_{2}\right) \right\rangle dt=\left\langle
p_5 (\top), v\left(  \top -\tau_{2}\right)
\right\rangle +\int_{0}^{ \top }\left\langle \dot{p}_5 (\tau) ,v\left( \tau -\tau_{2}\right) \right\rangle d\tau.
\end{equation}
By \eqref{equ10}, \eqref{equ20} and \eqref{equ30}, we obtain that
$$
\int_{0}^{\top}\left( \left\langle p_2 (t),v\left( t\right) \right\rangle
+\left\langle p_3 (t), v\left( t-\tau_{1}\right) \right\rangle
-\left\langle \dot{p}_4 (t), v\left( t\right) \right\rangle+\left\langle \dot{p}_5(t),
v\left(t-\tau_{2}\right) \right\rangle \right) dt
+\left\langle p_5 (t), v\left( \top-\tau_{2}\right)\right\rangle =0
$$
for all $v\left(\cdot\right) \in T\left( D,\bar{x}\left(\cdot\right) \right)$.
On the other hand,
\begin{equation*}
\begin{split}
\int_{0}^{\top}\left\langle p_3 \left( t\right), v\left( t-\tau_{1}\right) \right\rangle dt
&=\int_{-\tau_{1}}^{ \top -\tau_{1}}\left\langle p_3 \left(
\tau +\tau_{1}\right), v\left( \tau \right) \right\rangle d\tau\\
&=\int_{-\tau_{1}}^{0}\left\langle p_3 \left( \tau
+\tau_{1}\right), v\left( \tau \right) \right\rangle d\tau
+\int_{0}^{ \top -\tau_{1}}\left\langle p_3 \left(\tau
+\tau_{1}\right), v\left( \tau \right) \right\rangle d\tau \\
&=\int_{0}^{ \top -\tau_{1}}\left\langle p_3 \left( \tau
+\tau_{1}\right), v\left( \tau \right) \right\rangle d\tau
\end{split}
\end{equation*}
and
\begin{equation*}
\begin{split}
\int_{0}^{\top}\left\langle \dot{p}_5 \left( t\right)
,v\left( t-\tau_{2}\right) \right\rangle dt
&=\int_{-\tau_{2}}^{ \top -\tau_{2}}\left\langle \dot{p}_5\left( \tau
+\tau_{2}\right), v\left( \tau \right) \right\rangle d\tau \\
&=\int_{-\tau_{2}}^{0}\left\langle \dot{p}_5\left( \tau
+\tau_{2}\right) v\left( \tau \right) \right\rangle d\tau
+\int_{0}^{ \top -\tau_{2}}\left\langle \dot{p}_5\left( \tau
+\tau_{2}\right) ,v\left( \tau \right) \right\rangle d\tau \\
&=\int_{0}^{ \top -\tau_{2}}\left\langle \dot{p}_5\left( \tau
+\tau_{2}\right), v\left( \tau \right) \right\rangle d\tau.
\end{split}
\end{equation*}
Hence,
\begin{multline*}
\int_{0}^{\top}\left\langle p_2\left( t\right) -\dot{p}_4
\left( t\right), v\left( t\right) \right\rangle
dt+\int_{0}^{ \top -\tau_{1}}\left\langle p_3 \left(
t+\tau_{1}\right), v\left( t\right) \right\rangle
dt-\int_{0}^{ \top -\tau_{2}}\left\langle \dot{p}_5\left(
t+\tau_{2}\right), v\left( t\right) \right\rangle dt \\
+\left\langle p_5 \left( \top\right) ,
v\left( \top-\tau_{2}\right) \right\rangle =0
\end{multline*}
for all $v\left(\cdot\right) \in  T\left( D,\bar{x}\left(\cdot\right) \right)$.
Put
\begin{equation*}
\begin{split}
\bar{p_3}\left( t+\tau_{1}\right)
&=
\begin{cases}
p_3 \left( t+\tau_{1}\right) & \text{ if  } t \in \left[ 0, \top -\tau_{1}\right], \\
0 & \text{ if } t \in \, \left]\top -\tau_{1} ,\top \right],
\end{cases}\\
q\left( t+\tau_{2}\right) &=
\begin{cases}
\dot{p}_5\left( t+\tau_{2}\right) & \text{ if }
t\in \left[0, \top -\tau_{2}\right], \\
0 & \text{ if } t\in \left] \top -\tau_{2},\top \right].
\end{cases}
\end{split}
\end{equation*}
Then,
\begin{equation*}
\int_{0}^{\top}\left\langle p_2\left( t\right)
-\dot{p}_4\left( t\right) +\bar{p_3}\left( t+\tau_{1}\right)
-q\left(t+\tau_{2}\right), v\left( t\right) \right\rangle
dt+\left\langle p_5 \left( \top\right) ,
v\left( \top-\tau_{2}\right) \right\rangle =0
\end{equation*}
for all $v\left(\cdot\right) \in  T\left( D,\bar{x}\left(\cdot\right) \right)$.
In particular, for $v$ such that $v(\tau )=0$  for almost all $\tau \in \left[ -\tau_{1},0\right]$
and $v(\tau )=0$ for almost all $\tau \in \left[ \top-\tau_{2},0\right]$, we have
\begin{equation*}
p_2\left( t\right) -\dot{p}_4\left( t\right) +\bar{p_3}\left( t+\tau_{1}\right)
-q\left( t+\tau_{2}\right) =0 \quad \text{ a.e. } t \in \left[ 0,\top\right]
\end{equation*}
or
\begin{equation*}
\begin{cases}
\dot{p}_4\left( t\right) +\dot{p}_5\left( t+\tau_{2}\right)
=p_2\left( t\right)+p_3 \left( t+\tau_{1}\right)
& \text{ a.e. } t\in \left[ 0,\top -\tau_{1}\right], \\
\dot{p}_4\left( t\right) +\dot{p}_5\left( t+\tau_{2}\right)
=p_2\left( t\right)
& \text{ a.e. } t\in \, \left] \top -\tau_{1} ,\top -\tau_{2}\right],\\
\dot{p}_4\left( t\right)
= p_2\left( t\right)
& \text{ a.e. } t\in \, \left]\top -\tau_{2} ,\top \right].
\end{cases}
\end{equation*}
The proof is complete.
\end{proof}


\section{Optimal control with time delays}
\label{sec:oc:td}

Now we prove existence of an optimal solution
to more general problems of optimal control with
time delays. The result is obtained
via the exterior penalty method \cite{benharratM,34:Serova}
and Theorem~\ref{th:EL1}.
The optimal control problem with time delays is defined as follows:
\begin{equation}
\label{Pce}
\min \int_{0}^{\top }
l\left(t,x(t), \dot{x}\left( t-\tau_2\right), u(t)\right) dt
\end{equation}
subject to
\begin{equation}
\label{Pce1}
\dot{x}\left( t\right) =Ax\left( t-\tau_1\right)+Bu(t),
\quad t\in[0,\top] =:I,
\end{equation}
\begin{equation}
\label{Pce2}
x(t)=\theta_{1}(t), \quad t\in[-\tau_{1},-\tau_{2}] =: I_1,
\end{equation}
\begin{equation}
\label{Pce3}
x(t)=\theta_{2}(t), \quad t\in[-\tau_{2},0] =: I_2,
\end{equation}
and
\begin{equation}
\label{Pce4}
x(\top)=\alpha,
\end{equation}
where $x(\cdot)\in\mathcal{H}$, $u(\cdot)\in U_{0}=\left\{ u\left(\cdot\right)
\in L^{2}\left([0,\top],U\right) : u\left( 0\right) =0\right\}$, $A$ is an $N\times N$ matrix,
$B$ is an $N\times m$ matrix, and $l :[0, \top] \times \mathbb{R}^{N} \times \mathbb{R}^{N}\times \mathbb{R}^{m}
\rightarrow \mathbb{R}$, $\left( t,a,\bar{b},c\right) \rightarrow l\left( t,a,\bar{b},c\right)$.
The final time $\top > 0$ is fixed in $\mathbb{R}$, $\tau_{1}$ and $\tau_{2}$ are two given positive real
numbers such that $\tau_{2}< \tau_{1}<\top$ and, as before, $\theta_{1}(\cdot)$ and $\theta_{2}(\cdot)$
are given piecewise smooth functions. In the sequel, we denote by $\theta(\cdot)$ the function defined by
$\theta (t)=\theta_{1}(t)$,  $t\in I_1$, and $\theta(t)=\theta_{2}(t)$, $t\in I_2$.
We make the following assumptions on the data of the problem:
\begin{itemize}
\item[$(H_1)$] The mapping $l$ is a $C^{1}$-Carath\'{e}dory mapping, i.e.,
$l$ is $C^{1}$ in $\left(a,\bar{b},c \right)$ for almost all $t \in \left[ 0,T\right]$
and is measurable in $t$ for every  $\left( a,\bar{b},c \right)
\in  \mathbb{R}^{N}\times \mathbb{R}^{N}\times \mathbb{R}^{m}$;

\item[$(H_2)$] there exist $\gamma_i\left(\cdot\right)
\in  L^{2}\left( I,\mathbb{R}^{+}\right)$, $i=1, \ldots, 5$, such that
$$
\left\vert l\left( t,a,\bar{b},c\right) \right\vert
\leq \gamma_1 \left( t\right)
\quad \text{ and } \quad \left\Vert \partial_{i}
L\left( t,a,\bar{b},c\right) \right\Vert
\leq \gamma_{i} \left( t\right)  \text{ a.e. } t \in  I,
\ i=2, \ldots, 4,
$$
where $\partial_{i}L$ is the partial derivative of $l$
with respect to its $i$th argument, $i = 1, \ldots, 4$;
\item[$(H_3)$] there exists $\rho >0$ such that for almost all $t \in \left[ 0,T\right]$
and for all $\left(a,\bar{b},c \right) \in  \mathbb{R}^{N}\times \mathbb{R}^{N}\times \mathbb{R}^{m}$
$$
l\left(a,\bar{b},c \right)\geq \rho \left\| c \right\|_{\mathbb{R}^{m}};
$$
\item[$(H_4)$] $l\left(a,\bar{b},c\right)$ is convex in $(\bar{b},c)$.
\end{itemize}
Using the exterior penalty function method, we consider the following
sequence of unconstrained optimal control problems corresponding to
\eqref{Pce}--\eqref{Pce4}:
\begin{equation}
\label{P:n}
\tag{$\mathcal{P}_{n}$}
\begin{gathered}
\inf \int_{0}^{\top} l\left(t,x\left( t\right),
\dot{x}\left( t-\tau_2\right), u\left( t\right) \right) dt+\frac{c_{n}}{2}
\int_{0}^{\top}\left\Vert \dot{x}\left( t\right) - Ax\left(t-\tau_1\right)
-Bu\left( t\right)\right\Vert_{\mathbb{R}^{N}}^{2}dt,\\
x(t)=\theta_{1}(t) \ \text{ a.e. } t\in I_1,\\
x(t)=\theta_{2}(t) \  \text{ a.e. } t\in I_2,\\
x\left(\top\right) =\alpha,\\
x\left(\cdot\right) \in \mathcal{H},
\quad u\left(\cdot\right) \in U_{0},
\end{gathered}
\end{equation}
where $c_{n+1}\geq c_n$, $c_n \rightarrow \infty$. Denote
\begin{gather*}
L_{n}\left( t,a,\bar{a},b,\bar{b}, c\right) := l\left( t,a,\bar{b},c\right)
+\frac{c_{n}}{2}\left\Vert b-A\bar{a}+Bc
\right\Vert_{\mathbb{R}^{N}}^{2},\\
J_{n}\left( x\left(\cdot\right),u(\cdot)\right)
:= \int_{0}^{\top} L_{n}\left( t,x\left( t\right) ,x\left( t-\tau_1\right),
\dot{x}\left( t\right), \dot{x}\left( t-\tau_2 \right), u(t)\right),\\
D:=\left\{ x\left(\cdot\right) \in \mathcal{H}: x\left(\top\right)
=\alpha \right\}.
\end{gather*}
The sequence of unconstrained optimal control problems takes then the following form:
\begin{equation}
\label{P:n:2}
\tag{$\mathcal{P}_{n}$}
\begin{cases}
\inf J_{n}\left( x\left(\cdot\right),u(\cdot)\right),\\
x\left(\cdot\right) \in D,\\
u(\cdot)\in U_{0},
\end{cases}
\end{equation}
$n\in \mathbb{N}$.

\begin{lemma}
\label{l4a}
The cone of tangents $T\left( U_{0},u\left(\cdot\right) \right)$ is given by
$$
T\left( U_{0},u\left(\cdot\right) \right) =\left\{ w(\cdot)
\in L^{2}\left( I,\mathbb{R}^{m}\right) : w\left( 0\right)
=0\right\}.
$$
\end{lemma}

\begin{proof}
Similar to the proof of Lemma~\ref{l41}.
\end{proof}

It is well known that the penalty function method is a very effective technique for solving
constrained optimization problems via unconstrained ones. The main question is the convergence
of the sequence of solutions of the unconstrained optimal control problems to the original/constrained
problem. Before giving the convergence theorem, we begin with some preparatory results, which are
a direct consequence of the necessary optimality conditions given by Theorem~\ref{th:EL1}.

\begin{proposition}
\label{pro4.1}
For every $n$, if $\left( x_{n}\left(\cdot\right),u_{n}\left(\cdot\right)\right)
\in D \times  U_{0}$ is an optimal solution to $\left(\mathcal{P}_{n}\right)$, then
\begin{enumerate}
\item
$$
\begin{cases}
\frac{d}{dt}\phi _{n}\left( t\right) =A^{*} \phi_{n}\left( t
+\tau_1\right) +\frac{1}{c_{n}}a_{n}\left( t\right)
-\frac{1}{c_{n}}e_{n}\left( t\right)
& \text{ a.e. } t\in \left[ 0, \top -\tau_1\right],\\
\frac{d}{dt}\phi_{n}\left( t\right)
=\frac{1}{c_{n}}a_{n}\left( t\right)
-\frac{1}{c_{n}}e_{n}\left( t\right)
& \text{ a.e. } t\in \, \left]  \top -\tau_1,\top -\tau_2\right],\\
\frac{d}{dt}\phi_{n}\left( t\right) =\frac{1}{c_{n}}a_{n}\left(t\right)
& \text{ a.e. } t\in \, \left]  \top -\tau_2,\top \right], \\
B^{\star }\phi _{n}\left( t\right)
=\frac{1}{c_{n}} b_{n}\left( t\right)
& \text{ a.e. } t\in \left[ 0, \top \right],
\end{cases}
$$
where
\begin{equation*}
\begin{split}
\phi _{n}\left( t\right) &=\dot{x}_{n}\left( t\right)-A
x_{n}\left( t-\tau_1\right) +Bu_{n}\left( t\right),\\
a_{n}\left( t\right)
&=l_{a}^{\prime }\left( t,x_{n}\left( t\right),
\dot{x}_{n}\left( t-\tau_2\right),u_{n}\left( t\right) \right),\\
e_{n}\left( t\right)
&=l_{\bar{b}}^{\prime }\left( t,x_{n}\left( t\right),
\dot{x}_{n}\left( t-\tau_2\right),u_{n}\left( t\right) \right),\\
b_{n}\left( t\right)
&=l_{c}^{\prime }\left( t,x_{n}\left( t\right),
\dot{x}_{n}\left( t-\tau_2\right),u_{n}\left( t\right) \right);
\end{split}
\end{equation*}
\item there exists $M>0$ such that $\left\Vert \phi _{n}\left( t\right)
\right\Vert _{X}\leq M$ for all $t\in \left[0, \top\right]$
and all $n$ sufficiently large.
\end{enumerate}
\end{proposition}

\begin{proof}
1) Let $\left( x_{n}\left(\cdot\right) ,u_{n}\left(\cdot\right) \right) \in  D
\times  U_{0}$ be an optimal solution to $\left(\mathcal{P}_{n}\right)$.
Then, by Lemma~\ref{l41}, Lemma~\ref{l4a} and Theorem~\ref{th:EL1},
we obtain the necessary conditions of item 1 for problem $\left(\mathcal{P}_{n}\right)$.

2) Since  $u_{n}(0)=0$ and  $\dot{\theta}\left( 0^{+}\right) $ exists,   $\dot{x}_{n}\left( 0\right) $
is defined and  there exists $k >0$ such that  $\left\Vert \phi _{n}\left( 0\right)
\right\Vert \leq k$. By the first equation of item 1, we have
\begin{equation*}
\phi_{n}\left( t\right) =\phi _{n}\left( 0\right)
+A^{* }\int_{0}^{t} \phi _{n}\left( \tau
+\tau_1\right) d\tau +\frac{1}{c_{n}}\int_{0}^{t}\left(a_{n}\left(
\tau \right) +e_{n}\left(\tau \right)\right)d\tau,
\quad t\in \left[ 0,\top -\tau_1\right].
\end{equation*}
Consequently,
\begin{equation*}
\begin{split}
\left\Vert \phi _{n}\left( t\right) \right\Vert
&\leq \left\Vert
\phi_{n}\left( 0\right) \right\Vert +\left\Vert
A^{*}\right\Vert\int_{0}^{t} \left\Vert \phi _{n}\left(
\tau +\tau_1\right) \right\Vert d\tau +R_{n}\\
&\leq k+\alpha\int_{0}^{t} \left\Vert
\phi_{n}\left( \tau +\tau_1\right) \right\Vert d\tau +R_{n}
\end{split}
\end{equation*}
for all $t\in \left[ 0, \top -\tau_1\right]$
with $k=\left\Vert \phi _{n}\left( 0\right) \right\Vert$, $\alpha=\left\Vert
A^{*}\right\Vert$, and
$R_{n}=\frac{1}{c_{n}}\left(\left\Vert \gamma_2 \left(\cdot\right)\right\Vert_{L^{2}}
+\left\Vert \gamma_3 \left(\cdot\right) \right\Vert_{L^{2}}\right)$.
By Gronwall's lemma, we obtain that
\begin{equation}
\label{equ3.4}
\left\Vert \phi _{n}\left( t\right) \right\Vert \leq \left( k+R_{n}\right)
\exp \left(\alpha \left(\top - \tau_1\right)\right)
\quad  \text{ for all } t \in \left[0, \top -\tau_1\right].
\end{equation}
The second and third equalities of item 1 give
\begin{equation}
\label{equ3.5}
\left\Vert \phi _{n}\left( t\right) \right\Vert \leq \left\Vert
\phi_{n}\left(\top -\tau_1\right) \right\Vert +R_{n}
\quad \text{ for all } t\in \left] \top -\tau_1,\top\right].
\end{equation}
Now, the inequalities \eqref{equ3.4} and \eqref{equ3.5} imply that
$$
\left\Vert \phi _{n}\left( t\right) \right\Vert \leq M_{n}
\text{ for all } n \text{ and  for all } t\in \left[ 0,\top \right]
$$
with
$$
M_{n} =k\exp \left(\alpha \left(\top - \tau_1\right)\right)
+R_n (1+\exp \left(\alpha \left(\top - \tau_1\right) \right)).
$$
Since $R_{n}\rightarrow 0$, there exists $M>0$ such that
$$
\left\Vert \phi _{n}\left( t\right)
\right\Vert _{\mathbb{R}^N }\leq M
$$
for all $t\in \left[ 0, \top \right]$ and for all $n$ large.
\end{proof}

We are now ready to prove the convergence theorem, which reads as follows.

\begin{theorem}[Penalty convergence theorem]
\label{thm4.1}
If hypotheses $(H_1)$--$(H_4)$ hold and problem
\eqref{Pce}--\eqref{Pce4} has a finite value, then the sequence
$(x_{n}\left( \cdot \right) ,u_{n}\left(\cdot\right) )_{n}$ of solutions
to $\left( P_{n}\right)$ contains a subsequence
$(x_{k}\left(\cdot\right) ,u_{k}\left(\cdot\right) )_{k}$ such that
\begin{itemize}
\item $x_{k}\left(\cdot\right) \longrightarrow x\left(\cdot\right)$
strongly in $C\left( I,\mathbb{R}^{N}\right)$;

\item $u_{k}\left(\cdot\right) \longrightarrow u\left(\cdot\right)$
weakly in $L^{2}\left(I,\mathbb{R}^{m}\right)$;

\item $\dot{x}_{k}\left(\cdot\right) \longrightarrow \dot{x}\left(\cdot\right)$
weakly in $L^{2}\left(I,\mathbb{R}^{N}\right)$;
\end{itemize}
with $\left( x\left(\cdot\right), u\left(\cdot\right)\right)$
a solution to problem \eqref{Pce}--\eqref{Pce4}.
\end{theorem}

\begin{proof}
Let $\left(x_{n}\left(\cdot\right), u_{n}\left(\cdot\right) \right) \in  D
\times  U_{0}$ be an optimal solution to $\left(\mathcal{P}_{n}\right)$
for every $n$. By Proposition~\ref{pro4.1},
\begin{eqnarray*}
\left\Vert \dot{x}_{n}\left( t\right) \right\Vert &\leq &M+\left\|A\right\| \left\Vert
x_{n}\left( t-\tau_1\right) \right\Vert + \left\|B\right\| \left\|u_{n}\left( t\right)\right\|\\
&\leq &M +\beta \left\Vert x_{n}\left( t-\tau_1\right)
\right\Vert +\sigma \left\Vert u_{n}\left( t\right) \right\Vert.
\end{eqnarray*}
Because
\begin{equation*}
\begin{cases}
x_{n}\left( t\right) =\theta \left( 0\right)
+\int_{0}^{t} \dot{x}_{n}\left( \tau \right) d\tau
&\forall t\in \left[0, \top\right],\\
x_{n}\left( t\right) =\theta \left( t\right)
&\text{a.e. } t\in \left[ -\tau_1,0\right],
\end{cases}
\end{equation*}
it follows that
\begin{equation*}
\begin{split}
\left\Vert x_{n}\left( t\right) \right\Vert
&\leq \left\Vert \theta
\left( 0\right) \right\Vert +\int_{0}^{t}\left\Vert
\dot{x}_{n}\left( \tau \right) \right\Vert d\tau\\
&\leq \left\Vert \theta \left( 0\right)
\right\Vert +M \top + \beta \int_{0}^{t}\left\Vert x_{n}\left(
\tau -h\right) \right\Vert d\tau
+\sigma \int_{0}^{t}\left\Vert
u_{n}\left( \tau \right) \right\Vert d\tau.
\end{split}
\end{equation*}
On the other hand, if $\mathcal{M}$ denote the finite value of \eqref{Pce}--\eqref{Pce4}, then
\begin{equation*}
\int_{0}^{\top} l\left(t,x_{n}\left( t\right),
u_{n}\left( t\right) \right) dt
\leq J_{n}\left(x_{n}\left(\cdot\right) ,u_{n}(\cdot)\right) \leq \mathcal{M}.
\end{equation*}
By assumption $(H_5)$, there exists $K>0$ such that
\begin{equation*}
\left\Vert u_{n}\left(\cdot\right) \right\Vert_{L^{2}}\leq K.
\end{equation*}
Thus,
\begin{equation*}
\left\Vert x_{n}\left( t\right) \right\Vert \leq \left\Vert \theta \left( 0\right)
\right\Vert +M
\top +\sigma \top K + \beta
\int_{-\tau_1}^{0}\left\Vert \theta \left( \tau \right) \right\Vert d\tau
+\beta \int_{0}^{t-\tau_1}\left\Vert x_{n}\left( \tau \right)
\right\Vert d\tau.
\end{equation*}
By Gronwall's lemma, we obtain that
\begin{equation}
\label{equ3.33}
\left\Vert x_{n}\left( t\right) \right\Vert\leq \psi \quad \quad
 \text{ for   } n \text{ sufficiently large  and  for all } t\in \left[ 0, \top
\right],
\end{equation}
where
$$
\psi =\left( \left\Vert \theta \left( 0\right) \right\Vert
+M \top +\sigma \top K
+\beta \tau_1\left\Vert \theta \left(\cdot\right)
\right\Vert \right) \exp \left( \beta \left( \top
-\tau_1\right) \right).
$$
Similarly, for $n$ sufficiently large,
\begin{equation*}
\left\Vert \dot{x}_{n}\left( t\right) \right\Vert \leq M
+\beta \left\Vert x_{n}\left( t-\tau_1\right) \right\Vert +\sigma
\left\Vert u_{n}\left( t\right) \right\Vert.
\end{equation*}
For all  $t\in \left[ 0,h\right] $, we have
\begin{equation*}
\left\Vert \dot{x}_{n}\left( t\right) \right\Vert \leq M+\gamma \left(
t\right) +\beta \left\Vert \theta \left( t-\tau_1\right) \right\Vert +\sigma
\left\Vert u_{n}\left( t\right) \right\Vert =\omega \left( t\right).
\end{equation*}
Since $\omega \left(\cdot\right) \in L^{2}\left(I,\mathbb{R}\right)$
and $\left( u_{n}\left(\cdot\right) \right) _{n}$ is bounded in
$L^{2}\left(  I,\mathbb{R}^{m}\right)$, with $\left[ 0,h\right]$
of finite measure, there exists $\varrho >0$ such that
\begin{equation}
\label{equ3.6}
\left\Vert \dot{x}_{n}\left(\cdot\right) \right\Vert
\leq \varrho \ \ \ \ \text{ in } L^{2}\left( \left[ 0,\tau_1\right],  \mathbb{R}^N
\right), \text{ for } n\text{ sufficiently large}.
\end{equation}
For all  $t\in \left[ \tau_1, \top \right]$ we have
\begin{equation*}
\left\Vert \dot{x}_{n}\left( t\right) \right\Vert \leq M
+\gamma \left(t\right) +\beta \psi
+\sigma \left\Vert u_{n}\left( t\right) \right\Vert.
\end{equation*}
As before, we can assert that
\begin{equation}
\label{equ3.7}
\exists \hat{\varrho}>0\ :\ \left\Vert \dot{x}_{n}\left(\cdot\right)
\right\Vert \leq \hat{\varrho}
\quad \text{ in } L^{2}\left( \left[
\tau_1, \top\right] \mathbb{R}^N\right)
\text{ for } n \text{ sufficiently large}.
\end{equation}
By \eqref{equ3.6} and \eqref{equ3.7}, there exists  $\eta >0$ such that
\begin{equation*}
\left\Vert \dot{x}_{n}\left(\cdot\right) \right\Vert \leq \eta
\end{equation*}
in $L^{2}\left( \left[ 0, \top \right], \mathbb{R}^N\right)$
for $n$  sufficiently large. Therefore, there exists a subsequence
$\left( \dot{x}_{k}\left(\cdot\right)
\right) _{k}$ of $\dot{x}_{n}\left(\cdot\right)_{n}$
converging to
$\sigma \left(\cdot\right) \in L^{2}\left( I,\mathbb{R}^N\right)$.
Since $x_{n}\left( t\right) =\theta \left( 0\right) +\int_{0}^{t}
\dot{x}_{n}\left( \tau \right) d\tau $ for all $t\in I$,
by the use of  \eqref{equ3.33}, the sequence $\left( x_{n}\left(\cdot\right) \right)_{n}$
is equi-bounded and equi-continuous (because $\left( \dot{x}_{n}\left(\cdot\right) \right)_{n}$
is bounded in $ L^{2}\left(I,\mathbb{R}^N\right)$). Ascoli's theorem implies that
\begin{equation*}
x_{k}\left(\cdot\right) \longrightarrow x\left(\cdot\right)
\quad \text{ strongly in } C\left( I,\mathbb{R}^{N}\right).
\end{equation*}
Since
\begin{equation*}
x_{k}\left( 0\right) =\theta \left( 0\right) \text{ \ and \ }
\int_{0}^{t} \dot{x}_{n}\left( \tau \right) d\tau \longrightarrow
\int_{0}^{t}\sigma \left( \tau \right) d\tau,
\end{equation*}
we obtain that $x\left( t\right) =\theta \left( 0\right)
+\int_{0}^{t}\sigma \left( \tau \right) d\tau$ and $\dot{x}\left(t\right)
=\sigma \left( t\right)$ a.e. $t\in I$.
The sequence $\left( u_{n}\left(\cdot\right)\right)_{n}$ is bounded in
$L^{2}\left(  I,\mathbb{R}^{m}\right)$.
Thus, there exists a subsequence $\left(x_{k}\left(\cdot\right)\right)_k$
such that $u_{k}\left(\cdot\right) \longrightarrow u\left(\cdot\right)$
weakly in $L^{2}\left(I,\mathbb{R}^{N}\right)$. To complete
the proof, we show that $\left(x\left(\cdot\right),u\left(\cdot\right)\right)$
is an optimal solution to $\left(\mathcal{P}\right)$.
By Proposition~\ref{pro4.1}, we have
\begin{equation*}
B^{* } \phi _{k}\left( t\right) =\frac{1}{c_{k}}
b_{k}\left( t\right) \quad  \text{ a.e. } t\in I.
\end{equation*}
Hence,
$$
\int_{0}^{t}\left\Vert
B^{* } \phi_{k}\left( \tau \right) \right\Vert d\tau =\frac{1}{c_{k}}
\int_{0}^{t}\left\Vert b_{k}\left( \tau \right) \right\Vert d\tau
\leq \frac{1}{c_{k}}M
$$
with $M=\top\left\Vert \gamma_4 \left(\cdot\right) \right\Vert_{L^{2}}$.
We conclude that
\begin{equation*}
\int_{0}^{t} B^{*} \phi _{k}\left(\tau \right) d\tau
\longrightarrow 0\quad \text{ for all } t\in I.
\end{equation*}
On the other hand,
\begin{equation*}
\int_{0}^{t} B^{* } \phi _{k}\left(\tau \right) d\tau
\longrightarrow \int_{0}^{t} B^{*} \phi \left( \tau \right) d\tau
\quad \text{ for all } t\in I.
\end{equation*}
Consequently,
\begin{equation*}
\int_{0}^{t}B^{* } \phi \left(\tau\right) d\tau =0
\quad \text{ for all } t\in I.
\end{equation*}
This implies that
\begin{equation*}
B^{\star } \phi \left( t\right) =0
\quad \text{ for all } t\in I.
\end{equation*}
Thus,
\begin{gather*}
\phi \left( t\right) =0
\quad \text{ for all } t\in I,\\
\dot{x}\left( t\right) =Ax\left( t-\tau_1\right) )+Bu\left( t\right)
\quad \text{ for all } t\in I
\end{gather*}
and  $x\left( t\right) =\theta \left( t\right) $ a.e. $t\in \left[
-\tau_1,0\right]$, $x\left(  \top \right)=\alpha$.
Then, $\left(x\left(\cdot\right), u\left(\cdot\right)\right)$ is an admissible pair and
\begin{equation*}
\int_{I} l\left(t,x\left( t\right),\dot{x}\left( t-\tau_2\right) ,u(t)\right) \ge \mathcal{M}.
\end{equation*}
On the other hand,
\begin{equation*}
\int_{ I} l\left( t,x_{k}\left( t\right),
\dot{x}_{k}\left( t-\tau_2\right), u_{k}\left( t\right) \right) dt \leq \mathcal{M}.
\end{equation*}
Now the hypotheses $(H_1)$, $(H_2)$ and $(H_4)$, together with Lebesgue's theorem, assert that
\begin{equation*}
\int_{I} l\left(t,x\left( t\right),\dot{x}\left( t-\tau_2\right) ,u(t)\right)\leq \mathcal{M},
\end{equation*}
that is,
$$
\int_{I} l\left(t,x\left( t\right),\dot{x}\left( t-\tau_2\right) ,u(t)\right)= \mathcal{M}.
$$
This implies that the pair $\left(x\left(\cdot\right) ,u\left(\cdot\right)\right)$ is
a solution to problem \eqref{Pce}--\eqref{Pce4}.
\end{proof}


\section{Conclusion}
\label{sec:conc}

New optimality conditions for problems of the calculus of variations
and optimal control with time delays, where the delay in the unknown function
differs from the delay in its derivative/control, were obtained.
The proofs are first given in the simpler context of the delayed calculus of variations,
and then extended to delayed optimal control problems by using a penalty method.
New results include a convergence theorem (see Theorem~\ref{thm4.1}),
which is of great practical interest because it allows to obtain a solution
to a delayed optimal control problem by considering a sequence
of simpler problems of the calculus of variations.
Previous results in the literature \cite{MyID:256,MyID:231,MyID:253}
consider the delay in the unknown function to be the same
as the delay in its derivative. There is, however,
no justification for the delays to be the same.
In contrast with those results, here we consider
the case of multiple time delays.
Moreover, the procedure of our proofs is completely
different from the case of one time delay only,
which relies in the the Lagrange multiplier method.
Such approach introduces a new unknown function,
the Lagrange multiplier, for which it is hard to set
the interpolation space. Indeed, the Lagrange multiplier must be carefully
selected in order to be possible to obtain an accurate solution.
Otherwise, the resulting system of equations my become singular,
in particular if the number of degrees of freedom is too large.
Here we use a penalty method, which requires only the choice of one scalar parameter.
Big values of this parameter are used in order to impose the boundary conditions
in a proper manner. Furthermore, in our case the use of the penalty method
replaces a constrained optimization problem (the delayed optimal control problem)
by a sequence of unconstrained problems of the calculus of variations with time delay
whose solutions converge to the solution of the original constrained problem.
Similarly to \cite{MyID:231}, our results can be easily extended for controls with time delay.


\section*{Conflict of Interests}

The authors declare that there is no conflict of interests
regarding the publication of this paper.


\section*{Acknowledgements}

This work was partially supported by Portuguese funds through the
\emph{Center for Research and Development in Mathematics and Applications} (CIDMA),
and \emph{The Portuguese Foundation for Science and Technology} (FCT),
within project PEst-OE/MAT/UI4106/2014. Torres was also supported
by the FCT project PTDC/EEI-AUT/1450/2012, co-financed by FEDER under
POFC-QREN with COMPETE reference FCOMP-01-0124-FEDER-028894.
The authors are grateful to two anonymous referees
for valuable remarks and comments, which
significantly contributed to the quality of the paper.



\end{document}